\documentclass[3p,times]{elsarticle}
\usepackage{amsmath, amssymb, amstext}
\usepackage{fancyhdr}
\usepackage{algorithm}
\usepackage{algpseudocode}
\usepackage{mathtools}
\usepackage{theorem}
\usepackage{xcolor}
\usepackage{tikz}
\usetikzlibrary[topaths]
\usetikzlibrary{decorations.pathmorphing}
\usetikzlibrary{decorations.markings}

\newtheorem{fact}{Fact}[section]
\newtheorem{lemma}[fact]{Lemma}
\newtheorem{theorem}[fact]{Theorem}
\newtheorem{definition}[fact]{Definition}
\newtheorem{corollary}[fact]{Corollary}

\newenvironment{proof}{{\bf Proof:  }}{\hfill\rule{2mm}{2mm}}
 \newcommand{\R}{\mathbb{R}}

\newcommand{\cA}{\mathcal{A}} \newcommand{\cB}{\mathcal{B}}

\DeclareMathOperator{\convOp}{conv}
\newcommand{\conv}{\convOp}

\begin{document}
\begin{frontmatter}



\title{An Elementary Integrality
  Proof of Rothblum's Stable Matching Formulation}
\author[co]{Jochen K\"{o}nemann}
\author[co]{Kanstantsin Pashkovich}
\author[co]{Justin Toth\corref{cor1}}
\ead{wjtoth@uwaterloo.ca}
\address[co]{Department of Combinatorics and Optimization, University of Waterloo, Canada}

\cortext[cor1]{Corresponding author}

\begin{abstract}
  In this paper we provide a short new proof for the integrality of
  Rothblum's linear description of the convex hull of incidence
  vectors of stable matchings in bipartite graphs. The key feature of our proof is to show
  that extreme points of the formulation must have a $0,1$-component. 
\end{abstract}
\end{frontmatter}

\section{Introduction}

In an instance of the {\em stable marriage} problem, we are given a
bipartite graph $G=(\cA \cup \cB, E)$ where $\cA$ and $\cB$
traditionally represent sets of women and men, respectively. An edge
$ab \in E$ corresponds to an {\em acceptable} pair $a$ and $b$ of man
and woman. In the following, we let $N(u): = \{ v \,:\, uv \in E \}$ be
the set of neighbours of $u$ in $G$. Each node
$u \in V:=\cA \cup \cB$ specifies a complete {\em preference order}
$>_u$ over its neighbours where node $u$ prefers neighbour $v_1$
over $v_2$ iff $v_1 >_u v_2$. For ease of notation, we will think of $>_u$ as a total
ordering on $N(u) \cup \{\varnothing\}$ where 
$\varnothing$ is the least preferred element of each node $u \in V$. In their seminal paper
\cite{gale1962college}, Gale and Shapley introduced the above problem,
and provided a constructive proof of existence of so called {\em
  stable} matchings. A matching is a collection $M$ of edges in $E$
such that each node is incident to at most one edge in $M$. $M$ is
stable if, for every edge $uv \not\in M$, $M(u) >_u v$ or $M(v) >_v u$ where
$M(u)$ is the node matched to $u$ in $M$ if that exists, and
$M(u):=\varnothing$, otherwise. 

In this paper, we focus on polyhedral characterizations of the set of
incidence vectors of stable matchings.  Vande Vate first provided such
a description in \cite{vate1989linear} for the special case where $G$
is a complete bipartite graph.  Rothblum
\cite{rothblum1992characterization} later generalized Vande Vate's
result to incomplete preference lists and simplified the proof of
integrality.

We provide an even simpler, more compact argument for the integrality
of Rothblum's formulation. Our arguments are elementary and
rely solely on some well-known results on the symmetric difference
of stable matchings as well as some knowledge of the local structure
of extreme points in our formulation to achieve the desired result.

Necessary background from the literature will be covered in the section to follow. The main result, a one-page proof, will be given in section \ref{section:theorem}.

\section{Stable Matchings Preliminaries}

We briefly review a couple of well-known facts on stable matchings. 
For each edge $uv$ in $E$, we let
$\delta^{>u}(v):=\{ \{v,w\}\in E: w >_v u \}$ to be the set of edges
incident to $v$ and those of its neighbours that are preferred to $u$. 
For $v \in V$ let $N_{\max}(v)$ denote its most preferred neighbour.  
A matching $M$ in $G$ is now easily seen to be stable if 
\begin{equation}\label{eq:stability_def}
		M \, \cap \, \big(\delta^{>u}(v) \cup \delta^{>v}(u) \cup  \{e\} \big)\neq\varnothing\,,
\end{equation}
for all $e=uv \in E$. The following lemmas study the connected components of
the {\em symmetric difference} $M_1\triangle M_2:=(M_1\setminus M_2)
\cup (M_2 \setminus M_1)$ of 
stable matchings $M_1$ and $M_2$. Note that $M_1\triangle M_2$ is an edge set, however in this paper we refer to nontrivial connected components of the graph $(V, M_1\triangle M_2)$ as connected components of $M_1\triangle M_2$. Here, $V(C)$ and $E(C)$ denote
the set of nodes and edges, respectively, of a graph $C$.

\begin{lemma}\label{lemma:pref}
Let $M_1$ and $M_2$ be two stable matchings in $G$ and let $J$ be a
connected component in $M_1 \triangle M_2$. Then for some
$\{i,j\}=\{1,2\}$, we have
\begin{equation}\label{eq:pref}
	M_i(a)>_a M_j(a) ~~\text{and} ~~  M_j(b)>_b M_i(b)\,,\,
\end{equation}
for all $a \in V(J) \cap \cA$ and $b \in V(J) \cap \cB$.
\end{lemma}
\begin{proof}
  Since $M_1$ and $M_2$ are matchings, $J$ is a path or a cycle with edges alternating between $M_1$ and $M_2$. Let
  $v\in V$ be an end node of $J$ if $J$ is a path, otherwise let $v$
  be an arbitrary node of $J$. For visualization of the proof see Figure~\ref{fig:lemma_pref}.

  W.l.o.g. $a:=v,\ a\in \mathcal{A}$ and $b:=M_1(a) >_a M_2(a)$. If
  $a=M_1(b) >_b M_2(b)$, the matching $M_2$
  violates~\eqref{eq:stability_def} for the edge $ab\in E$. Thus,
  $M_2(b) >_b M_1(b)=a$. Thus, $M_2(b)\neq \varnothing$ and the
  matching $M_1$ satisfies~\eqref{eq:stability_def} for the edge
  $e:=bM_2(b)\in E$ only if $M_1(M_2(b))>_{M_2(b)} b$. Continuing
  in this way, we obtain statement~\eqref{eq:pref}.
\end{proof}

\begin{figure}[H]
\centering
\begin{tikzpicture}[yscale=.9,xscale=1.2]

%% first subfigure

%% nodes

\node[shape=circle,draw=black,inner sep=0pt, minimum size=4pt, label=left:$a$] (a1) at (0,4) {};
\node[shape=circle,draw=black,inner sep=0pt, minimum size=4pt, label=right:$b$](b1) at (1.5,4){};
\node[shape=circle,draw=black,inner sep=0pt, minimum size=4pt, label=left:$a_1$] (a2) at (0,2) {};
\node[shape=circle,draw=black,inner sep=0pt, minimum size=4pt, label=right:$b_1$](b2) at (1.5,2){};

%% lines
\path[-,line width=1pt, blue] (a1) edge  (b1);
\path[-,line width=1pt,decoration={zigzag,segment length=3,amplitude=.5,post=lineto,post length=2pt}, red] (a1) edge[decorate] (b2);
\path[-,line width=1pt, blue] (a2) edge (b2);
\path[-,line width=1pt,decoration={zigzag,segment length=3,amplitude=.5,post=lineto,post length=2pt}, red] (a2) edge[decorate] (b1);

%% special line
\path[-,line width=6pt, blue, opacity=.15] (a1) edge  (b1);
  
%% arrows 
\path[-,line width=1pt, decoration={markings, mark=at position .3 with {\arrow{latex}}}, blue] (a1) edge[decorate] (b1);

%% second subfigure  
  
\begin{scope}[xshift=3.5cm]

\node[](c) at (-1,3) {$\implies$};

%% nodes

\node[shape=circle,draw=black,inner sep=0pt, minimum size=4pt, label=left:$a$] (a1) at (0,4) {};
\node[shape=circle,draw=black,inner sep=0pt, minimum size=4pt, label=right:$b$](b1) at (1.5,4){};
\node[shape=circle,draw=black,inner sep=0pt, minimum size=4pt, label=left:$a_1$] (a2) at (0,2) {};
\node[shape=circle,draw=black,inner sep=0pt, minimum size=4pt, label=right:$b_1$](b2) at (1.5,2){};

%% lines
\path[-,line width=1pt, blue] (a1) edge  (b1);
\path[-,line width=1pt,decoration={zigzag,segment length=3,amplitude=.5,post=lineto,post length=2pt}, red] (a1) edge[decorate] (b2);
\path[-,line width=1pt, blue] (a2) edge (b2);
\path[-,line width=1pt,decoration={zigzag,segment length=3,amplitude=.5,post=lineto,post length=2pt}, red] (a2) edge[decorate] (b1);

%% special line
\path[-,line width=6pt,decoration={zigzag,segment length=3,amplitude=.5,post=lineto,post length=2pt}, red, opacity=.15] (a2) edge[decorate]  (b1);  
  
%% arrows 

\path[-,line width=1pt, decoration={markings, mark=at position .25 with {\arrow{latex}}}, blue] (a1) edge[decorate] (b1);
\path[-,line width=1pt, decoration={markings, mark=at position .2 with {\arrow{latex}}}, red] (b1) edge[decorate] (a2);
\end{scope}

%% third subfigure  
  
\begin{scope}[xshift=7cm]

\node[](c) at (-1,3) {$\implies$};

%% nodes

\node[shape=circle,draw=black,inner sep=0pt, minimum size=4pt, label=left:$a$] (a1) at (0,4) {};
\node[shape=circle,draw=black,inner sep=0pt, minimum size=4pt, label=right:$b$](b1) at (1.5,4){};
\node[shape=circle,draw=black,inner sep=0pt, minimum size=4pt, label=left:$a_1$] (a2) at (0,2) {};
\node[shape=circle,draw=black,inner sep=0pt, minimum size=4pt, label=right:$b_1$](b2) at (1.5,2){};

%% lines
\path[-,line width=1pt, blue] (a1) edge  (b1);
\path[-,line width=1pt,decoration={zigzag,segment length=3,amplitude=.5,post=lineto,post length=2pt}, red] (a1) edge[decorate] (b2);
\path[-,line width=1pt, blue] (a2) edge (b2);
\path[-,line width=1pt,decoration={zigzag,segment length=3,amplitude=.5,post=lineto,post length=2pt}, red] (a2) edge[decorate] (b1);

%% special line
\path[-,line width=6pt, blue, opacity=.15] (a2) edge  (b2);
  
%% arrows 

\path[-,line width=1pt, decoration={markings, mark=at position .25 with {\arrow{latex}}}, blue] (a1) edge[decorate] (b1);
%\path[-,line width=1pt, decoration={markings, mark=at position .2 with {\arrow{latex}}}, red] (b2) edge[decorate] (a1);
\path[-,line width=1pt, decoration={markings, mark=at position .25 with {\arrow{latex}}}, blue] (a2) edge[decorate] (b2);
\path[-,line width=1pt, decoration={markings, mark=at position .2 with {\arrow{latex}}}, red] (b1) edge[decorate] (a2);
\end{scope}

%% forth subfigure  
  
\begin{scope}[xshift=10.5cm]

\node[](c) at (-1,3) {$\implies$};

%% nodes

\node[shape=circle,draw=black,inner sep=0pt, minimum size=4pt, label=left:$a$] (a1) at (0,4) {};
\node[shape=circle,draw=black,inner sep=0pt, minimum size=4pt, label=right:$b$](b1) at (1.5,4){};
\node[shape=circle,draw=black,inner sep=0pt, minimum size=4pt, label=left:$a_1$] (a2) at (0,2) {};
\node[shape=circle,draw=black,inner sep=0pt, minimum size=4pt, label=right:$b_1$](b2) at (1.5,2){};

%% lines
\path[-,line width=1pt, blue] (a1) edge  (b1);
\path[-,line width=1pt,decoration={zigzag,segment length=3,amplitude=.5,post=lineto,post length=2pt}, red] (a1) edge[decorate] (b2);
\path[-,line width=1pt, blue] (a2) edge (b2);
\path[-,line width=1pt,decoration={zigzag,segment length=3,amplitude=.5,post=lineto,post length=2pt}, red] (a2) edge[decorate] (b1);

%% arrows 

\path[-,line width=1pt, decoration={markings, mark=at position .25 with {\arrow{latex}}}, blue] (a1) edge[decorate] (b1);
\path[-,line width=1pt, decoration={markings, mark=at position .2 with {\arrow{latex}}}, red] (b2) edge[decorate] (a1);
\path[-,line width=1pt, decoration={markings, mark=at position .25 with {\arrow{latex}}}, blue] (a2) edge[decorate] (b2);
\path[-,line width=1pt, decoration={markings, mark=at position .2 with {\arrow{latex}}}, red] (b1) edge[decorate] (a2);
\end{scope}

\end{tikzpicture}
\caption{Visualizing the proof of Lemma \ref{lemma:pref}. Here, the connected component $J$ is a cycle on four nodes. The edges of $M_1$ are marked by straight blue lines, and the edges of $M_2$ by red zigzags. So here $a_1 = M_2(b)$ and $b_1 = M_1(M_2(b))$. For each node $w\in \{a,b,a_1,b_1\}$, the arrow at $w$ points towards the most preferred node in $\{M_1(w), M_2(w)\}$ with respect to $>_w$. }
\label{fig:lemma_pref}
\end{figure}
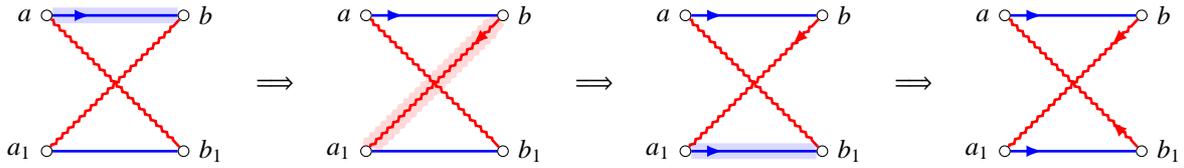

The next Lemma is equivalent to Theorem 2.16 in \cite{roth1992two}. In the interest of self-containment we provide a short elementary proof below.
\begin{lemma}\label{lemma:sym_stable} 
  Let $M_1$ and $M_2$ be two stable matchings in $G$. Let $J_1$ be
  those connected components of $M_1 \triangle M_2$ 
  that satisfy
  \eqref{eq:pref} for $i=1$ and $j=2$ (i.e., $\cA$ nodes prefer
  $M_1$ edges); let $J_2$ be all
  remaining connected components of $M_1 \triangle M_2$. 
  Then both $M'_1 = M_1\triangle E(J_1)$ and
  $M'_2=M_1\triangle E(J_2)$ are stable matchings in $G$.
\end{lemma}

\begin{proof}
  For contradiction assume
  that one of the matchings $M'_1$ and $M'_2$ is not stable; w.l.o.g. assume that
  $M'_1$ does not satisfy~\eqref{eq:stability_def} for some edge
  $ab\in E$ with $a\in\mathcal{A}$ and $b\in\mathcal{B}$. For visualization of the proof see Figure~\ref{fig:sym_stable}.

  Since $M_1$ and $M_2$ are stable, $a$ and $b$ lie in $V(J_1)\cup V(J_2)$. Otherwise $M'_1(a)=M_1(a)$,  $M'_1(b)=M_1(b)$ or  $M'_1(a)=M_2(a)$,  $M'_1(b)=M_2(b)$, and thus one of $M_1, M_2$ also violates~\eqref{eq:stability_def} for edge $ab$. Similarly, $a$ and $b$ cannot both lie in $V(J_1)$ or both in  $V(J_2)$. Suppose first that $a \in V(J_1)$ and $b \in V(J_2)$. In this case $M'_1(a)=M_2(a)$ and
  $M'_1(b) = M_1(b) >_b M_2(b)$. Thus $M_2$
  violates~\eqref{eq:stability_def} for edge $ab$. 

  If $a\in V(J_2)$
  and $b\in V(J_1)$, then $M'_1(a)=M_1(a)$ and $M'_1(b)>_b M_1(b)$,
  and hence $M_1$ violates~\eqref{eq:stability_def} for edge $ab$,
  contradiction.
\end{proof}

\begin{figure}[H]
\centering
\begin{tikzpicture}

% nodes

\node[shape=circle,draw=black,inner sep=0pt, minimum size=4pt, label=left:$b_1$] (a1) at (0,4) {};
\node[shape=circle,draw=black,inner sep=0pt, minimum size=4pt, label=right:$a_1$](b1) at (2,4){};
\node[shape=circle,draw=black,inner sep=0pt, minimum size=4pt, label=left:$b_2$] (a2) at (0,2) {};
\node[shape=circle,draw=black,inner sep=0pt, minimum size=4pt, label=right:$a$](b2) at (2,2){};
\node[shape=circle,draw=black,inner sep=0pt, minimum size=4pt, label=left:$b$] (a3) at (5,4) {};
\node[shape=circle,draw=black,inner sep=0pt, minimum size=4pt, label=right:$a_3$](b3) at (7,4){};
\node[shape=circle,draw=black,inner sep=0pt, minimum size=4pt, label=left:$b_4$] (a4) at (5,2) {};
\node[shape=circle,draw=black,inner sep=0pt, minimum size=4pt, label=right:$a_4$](b4) at (7,2){};

%% lines
\path[-,line width=1pt, blue] (a1) edge  (b1);
\path[-,line width=1pt,decoration={zigzag,segment length=3,amplitude=.5,post=lineto,post length=2pt}, red] (a1) edge[decorate] (b2);
\path[-,line width=1pt, blue] (a2) edge (b2);
\path[-,line width=1pt,decoration={zigzag,segment length=3,amplitude=.5,post=lineto,post length=2pt}, red] (a2) edge[decorate] (b1);

\path[-,line width=1pt, blue] (a3) edge  (b3);
\path[-,line width=1pt,decoration={zigzag,segment length=3,amplitude=.5,post=lineto,post length=2pt}, red] (a3) edge[decorate] (b4);
\path[-,line width=1pt, blue] (a4) edge (b4);
\path[-,line width=1pt,decoration={zigzag,segment length=3,amplitude=.5,post=lineto,post length=2pt}, red] (a4) edge[decorate] (b3);
  
%% arrows 

\path[-,line width=1pt, decoration={markings, mark=at position .25 with {\arrow{latex}}}, blue] (b1) edge[decorate] (a1);
\path[-,line width=1pt, decoration={markings, mark=at position .2 with {\arrow{latex}}}, red] (a1) edge[decorate] (b2);
\path[-,line width=1pt, decoration={markings, mark=at position .25 with {\arrow{latex}}}, blue] (b2) edge[decorate] (a2);
\path[-,line width=1pt, decoration={markings, mark=at position .2 with {\arrow{latex}}}, red] (a2) edge[decorate] (b1);

\path[-,line width=1pt, decoration={markings, mark=at position .25 with {\arrow{latex}}}, blue] (a3) edge[decorate] (b3);
\path[-,line width=1pt, decoration={markings, mark=at position .2 with {\arrow{latex}}}, red] (b4) edge[decorate] (a3);
\path[-,line width=1pt, decoration={markings, mark=at position .25 with {\arrow{latex}}}, blue] (a4) edge[decorate] (b4);
\path[-,line width=1pt, decoration={markings, mark=at position .2 with {\arrow{latex}}}, red] (b3) edge[decorate] (a4);

%% symmetric difference M_1 and J_1

%% special lines
\path[-,line width=2pt,  opacity=.15] (b2) edge  (a3);

\path[-,line width=6pt,  opacity=.15,decoration={zigzag,segment length=3,amplitude=.5,post=lineto,post length=2pt}, red] (a1) edge[decorate] (b2);
\path[-,line width=6pt,  opacity=.15,decoration={zigzag,segment length=3,amplitude=.5,post=lineto,post length=2pt}, red] (a2) edge[decorate] (b1);
\path[-,line width=6pt,  opacity=.15, blue] (a3) edge  (b3);
\path[-,line width=6pt,  opacity=.15, blue] (a4) edge (b4);

%% J_1, J_2
\draw[very thick, opacity=0.2, dashed, rounded corners=20pt] (-.7,1.3)rectangle (2.7,4.7);
\draw[very thick, opacity=0.2, dashed, rounded corners=20pt] (4.3,1.3)rectangle (7.7,4.7);
\node[] (J1) at (-1.3,4) {$J_1$};
\node[] (J2) at (8.3,4) {$J_2$};

\end{tikzpicture}
\caption{Visualizing the proof of Lemma \ref{lemma:sym_stable}.  Here, both $J_1$ and $J_2$ consist of one cycle on four nodes. The edges of $M_1$ are marked by blue straight lines, and the edges of $M_2$ by red zigzags; the edges of $M_1\triangle  E(J_1)$ are the highlighted edges of $M_1$ and $M_2$. For each node $w$, the arrow at $w$ points towards the most preferred node in $\{M_1(w), M_2(w)\}$ with respect to $>_w$. 
The figure illustrates the case, when the matching $M'_1=M_1\triangle E(J_1)$ violates~\eqref{eq:stability_def} for the edge $\{a,b\}$ with $a\in V(J_1)$
  and $b\in V(J_2)$. So here $M'_1(a)=M_2(a)=b_1$ and $M'_1(b)=M_1(b)=a_3>_{b} a_4=M_2(b)$. In this case, $M_2$ violates~\eqref{eq:stability_def} for the edge $\{a,b\}$ as well, contradiction.}
\label{fig:sym_stable}
\end{figure}
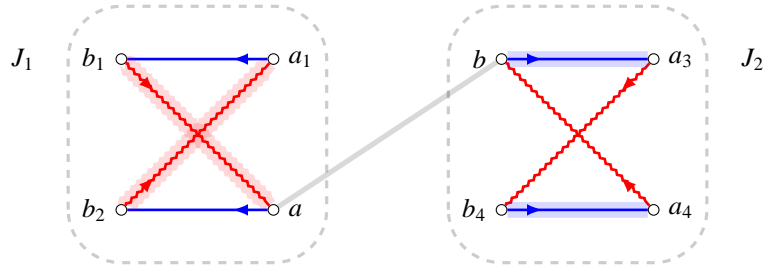

%\section{Stable Matching Polytope}

\begin{definition}Let us define the \emph{stable matching polytope} $P(G)\subseteq\R^E$ for graph $G$ as follows
$$
	P(G):=\conv\{\chi(M)\in\R^E: M \text{ is a stable matching in } G\}\,.
$$
By~\cite{gale1962college}, $P(G)$ is a nonempty polytope because every
graph $G$ has a stable matching.
\end{definition}
Clearly, the vertices of $P(G)$ are in one-to-one correspondence with
stable matchings in $G$. Moreover, Lemma~\ref{lemma:sym_stable} helps
to understand what pairs of stable matchings in $G$ do not correspond
to edges of $P(G)$.

\begin{lemma}\label{lemma:edge}
  Let $M_1$ and $M_2$ be two stable matchings in $G$ which define an
  edge of the polytope~$P(G)$. 
  Then all connected components in
  $M_1\triangle M_2$ satisfy~\eqref{eq:pref} for unique choice of $i$
  and $j$. 
\end{lemma}
\begin{proof}
  Suppose for contradiction that the statement of the lemma does not
  hold. Hence the sets $J_1$ and $J_2$ are both nonempty in
  Lemma~\ref{lemma:sym_stable}, and we obtain stable matchings
  $M_1\triangle E(J_1)$ and $M_1\triangle E(J_2)$ that are different from $M_1$,
  $M_2$. 
  We also have 
  \[ \frac{1}{2}\chi\big(M_1\triangle E(J_1)\big)+\frac{1}{2}\chi\big(M_1\triangle
  E(J_2)\big) =\frac{1}{2}\chi\big(M_1\big)+\frac{1}{2}\chi\big(M_2\big),\]
  and hence there are two distinct convex combinations of the midpoint
  of the edge between $M_1$ and $M_2$; a contradiction. 
\end{proof}

The next Corollary can be obtained from Ratier's characterization of edges of the stable matching polytope~\cite{ratier1996stable}.

\begin{corollary}\label{cor:edge}
Let $M_1$ and $M_2$ be two stable matchings in $G$ such that
\begin{equation}\label{eq:edge}
M_1\cap\delta^{>a}(b)\neq\varnothing,\
M_1\cap\delta^{>b}(a)\neq\varnothing\quad\text{and}\quad
M_2\cap\big(\delta^{>a}(b)\cup \delta^{>b}(a)\big)=\varnothing
\end{equation}
 for some $a\in\mathcal{A}$ and $b\in\mathcal{B}$. Then, $M_1$ and $M_2$ do not define an edge of the polytope~$P(G)$.
\end{corollary}
\begin{proof}
  Condition \eqref{eq:edge} implies that both $a$ and $b$ prefer
  $M_1$ over $M_2$. Hence, both sets $J_1$ and $J_2$ as given in Lemma
  \ref{lemma:sym_stable} must be non-empty. An application of Lemma
  \ref{lemma:edge} completes the proof of the corollary.
\end{proof}

\section{Linear Description}\label{section:theorem}
Let us define $Q(G)\subseteq\R^E$ to be the polytope described by the
following linear constraints
\begin{align}
  x(\delta(v)) \leq 1\,,\quad \forall v \in V\qquad \text{and} \qquad x_e \geq 0\,,\quad \forall e \in E\,,\label{eq:lin_descr_match}\\
  x(\delta^{>a}(b))+ x(\delta^{>b}(a)) + x_{ab} \geq 1\,, \quad \forall ab \in E \label{eq:lin_descr_stab}
\end{align}
where $x(J) := \sum_{e \in J} x_e$ for any $J \subseteq E$.

Clearly, $P(G)\subseteq Q(G)$ because for every stable matching $M$ in
$G$ the point $x:=\chi(M)$ satisfies~\eqref{eq:lin_descr_match} and
by~\eqref{eq:stability_def} the point $x$ also
satisfies~\eqref{eq:lin_descr_stab}. On the other hand, every integral
point in $Q(G)$ equals $\chi(M)$ for some stable matching $M$
in~$G$. In the remaining part of the paper we show that every vertex
of $Q(G)$ is integral, thus proving the main theorem.

\begin{theorem}
	For every graph $G$ the polytope $P(G)$ equals $Q(G)$.
\end{theorem}

\begin{lemma}
	For every graph $G$ every vertex of the polytope $Q(G)$ is integral.
\end{lemma}
\begin{proof}
  We first claim that every vertex $x$ of $Q(G)$ satisfies $x_e
  \in \{0,1\}$ for at least one $e \in E$. Assume for contradiction
  that $0 < x_e < 1$ for all $e \in E$. 
  Like every vertex of $Q(G)$, $x$ is uniquely defined by $|E|$ linearly independent tight constraints describing $Q(G)$. 
  Since $x$ has no zero coordinate, we can assume
  that these constraints are the constraints $x(\delta(v))\le 1$ for $v\in V_x$ 
  and the constraints~\eqref{eq:lin_descr_stab} for $e\in E_x$, where
  $|V_x|+|E_x|=|E|$. Moreover, let us assume that we choose the $|E|$
  tight constraints so that $|V_x|$ is as large as possible.

The constraints $x(\delta(v))=1$, $v\in V$ are linearly dependent, since $\sum_{a\in\mathcal{A}}x(\delta(a))=\sum_{b\in\mathcal{B}}x(\delta(b))$ for every $x \in \R^E$. Hence, we have $V_x\subsetneq V$. On the other hand if $a=N_{\max}(b)$ for some $b \in V$ then $e:=ab\not\in E_x$. Indeed, $a=N_{\max}(b)$ implies $\delta^{>a}(b)=\varnothing$, then
$$
	1\le x(\delta^{>a}(b))+ x(\delta^{>b}(a)) + x_{ab}=x(\delta^{\ge b}(a)) \le x(\delta(a)) \le 1\,,
$$
showing that $\delta^{< b}(a)=\varnothing$. So by linear independence we cannot have both $e \in E_x$ and $a \in V_x$. Suppose for a contradiction $e \in E_x$. Then $a \not\in V_x$ and hence $|E_x \setminus \{e\}| + |V_x \cup \{a\}| = |E|$, and $|V_x \cup \{a\}| > |V_x|$. Moreover $E_x\setminus \{e\}$ and $V_x\cup \{a\}$ also define the vertex $x$. This contradicts the choice of $V_x$, $E_x$. Analogously, we can show that if $b\in V_x$  and $a=N_{\min}(b)$ then $e:=ab\not\in E_x$. Moreover, notice that $N_{\min}(v)\neq N_{\max}(v)$ for $v\in V_x$ since no coordinate of $x$ equals~$1$. Thus, 
\[
	|E_x|=\frac{1}{2}\sum_{v\in V} |\delta(v)\cap E_x|\le \frac{1}{2}\sum_{v\in V_x} (|\delta(v)|-2)+\frac{1}{2}\sum_{v\in V\setminus V_x} (|\delta(v)|-1)= |E|-|V_x|-\frac{1}{2}|V\setminus V_x|\,,
\]
which implies $|E_x|+|V_x|< |E|$, contradiction.

\bigskip

Now let us assume that $G$ is a graph with the minimum number of edges such that $Q(G)$ is not an integral polytope. Let $x$ be a non-integral vertex of $Q(G)$.

Case $x_{ab}=0$ for some $a\in \mathcal{A}$, $b\in\mathcal{B}$ and
$e:=ab\in E$. In this case, let $P'$ and $x'$ be obtained from $Q(G)\cap\{x\in\R^E: x_{ab}=0\}$ and $x$ by dropping
the coordinate corresponding to $ab$. Then,
$x'$ is a vertex of the polytope $P'$, as otherwise $x$ is not a vertex of $Q(G)$.  Let $G'$ be the graph with
$V(G')=V$ and $E(G') = E \setminus \{e\}$. Then 
$$P' = P(G') \cap \{x\in \R^{E(G')}: x(\delta^{>a}(b))+ x(\delta^{>b}(a))\geq1\}\,,$$
since $P(G')=Q(G')$ by our minimality assumption. Define $H'$ to be the hyperplane $\{x\in \R^{E(G')}: x(\delta^{>a}(b))+ x(\delta^{>b}(a))=1\}$.  Then every vertex
of $P'$ is either a vertex of $P(G')$ or the intersection of an edge
of $P(G')$ with the hyperplane $H'$. Since the vertices of $P(G')$ are
integral, it remains to consider
vertices of $P'$ at the intersection of $H'$ and an edge of $P(G')$. Such an
edge would be defined by distinct stable matchings $M_1$ and $M_2$, where the vertex of $P'$ under consideration is not $\chi(M_1)$ nor $\chi(M_2)$. Note, that none of $\chi(M_1)$, $\chi(M_2)$ lies on the hyperplane $H'$, since $x'$ is the unique common point of $H'$ and the line segment between $\chi(M_1)$ and $\chi(M_2)$. Thus, $|M_1 \cap \big(\delta^{>a}(b)\cup\delta^{>b}(a)\big)|\neq 1$ and $|M_2 \cap \big(\delta^{>a}(b)\cup\delta^{>b}(a)\big)|\neq 1$. On the other hand, the line segment between $\chi(M_1)$ and $\chi(M_2)$ has a nonempty intersection with $H'$, so w.l.o.g. we may assume that $|M_1 \cap \big(\delta^{>a}(b)\cup\delta^{>b}(a)\big)|=2$ and $|M_2 \cap \big(\delta^{>a}(b)\cup\delta^{>b}(a)\big)|=0$. Therefore, $M_1$ and $M_2$ satisfy \eqref{eq:edge} for the given edge $ab$. So
Corollary \ref{cor:edge} readily implies that $P(G')$ cannot
have an edge connecting $M_1$ and $M_2$.

Case $x_{ab}=1$ for some $a\in \mathcal{A}$, $b\in\mathcal{B}$. Let
$x'$ be obtained by dropping the coordinates corresponding to
$\delta(a)\cup\delta(b)$, and let $G'$ be the graph with
$V(G')=V\setminus\{a,b\}$ and
$E(G') = E \setminus \big(\delta(a)\cup\delta(b)\big)$. It is
straightforward to see that $x'$ is a vertex of $Q(G')$. Due to minimality assumption $P(G')=Q(G')$ and thus both $x'$ and $x$ are integral, a contradiction.
\end{proof}

\bibliographystyle{alpha}
\bibliography{StableMatchingPolytopeReferences}
\end{document}